\documentclass{amsart}
\usepackage{amsmath,amsthm,indentfirst}
\usepackage{amssymb}
\usepackage{amsfonts}
\usepackage{a4}
\usepackage{amscd}
\usepackage[latin1]{inputenc}
\usepackage{url}
\DeclareMathAlphabet{\mathpzc}{OT1}{pzc}{m}{it}

\theoremstyle{plain}
\newtheorem{thm}{Theorem}
\newtheorem{cor}[thm]{Corollary}
\newtheorem{lem}[thm]{Lemma}
\newtheorem{prop}[thm]{Proposition}

\newtheorem{definition}[thm]{Definition}

\newtheorem{remark}[thm]{Remark}

\def\bbz{\mathbb{Z}}
\def\bbq{\mathbb{Q}}

\def\bbr{\mathbb{R}}
\def\bba{\mathbb{A}}

\def\bbh{\mathbb{H}}
\def\bbn{\mathbb{N}}
\def\bbg{\mathbb{G}}
\def\bbt{\mathbb{T}}

\def\bbu{\mathbb{U}}
\def\bbp{\mathbb{P}}

\def\tbbg{\tilde{\mathbb{G}}}
\def\gcal{\mathcal{G}}

\def\ucal{\mathcal{U}}

\def\ocal{\mathcal{O}}
\def\ocal{\mathcal{O}}
\def\scal{\mathcal{S}}

\def\dcal{\mathcal{D}}

\def\hcal{\mathcal{H}}

\def\pcal{\mathcal{P}}
\def\tgcal{\tilde{\mathcal{G}}}

\def\pfr{\mathfrak{P}}

\def\pfr{\mathfrak{p}}

\def\f{\mathfrak{f}}

\def\xbf{\mathbf{x}}

\def\vbf{\mathbf{v}}

\DeclareMathOperator\Spec{Spec}
\DeclareMathOperator\Aut{Aut}

\DeclareMathOperator\SL{SL}
\DeclareMathOperator\GL{GL}

\DeclareMathOperator\Lie{Lie}

\DeclareMathOperator\diag{diag}
\DeclareMathOperator\Zcl{Zcl}
\DeclareMathOperator\Hom{Hom}

\def\h{\hspace{1mm}}

\def\vare{\varepsilon}

\begin{document}
\title{Affine Sieve}
\author{Alireza Salehi Golsefidy}
\address{Mathematics Dept, University of California, San Diego, CA 92093-0112}
\email{asalehigolsefidy@ucsd.edu}
\author{Peter Sarnak}
\address{Mathematics Dept, Princeton University, Princeton, NJ 08544-1000}
\email{sarnak@math.princeton.edu}
\thanks{A. S-G. was partially supported by the NSF grant DMS-0635607 and NSF grant DMS-1001598 and P. S. by an NSF grant.}
\subjclass{20G35, 11N35}
\date{9/12/2011}
\begin{abstract}
We establish the main saturation conjecture in \cite{BGS} connected with executing a Brun sieve in the setting of an orbit of a group of affine linear transformations. This is carried out under the condition that the Zariski closure of the group is Levi-semisimple. It is likely that this condition is also necessary for such saturation to hold. 
\end{abstract}
\maketitle
\section{Introduction}
The purpose of this note is to complete the program initiated in \cite{BGS} of developing a Brun combinatorial sieve in the context of a group of affine linear motions. As explained below, this is possible in part thanks to recent developments concerning expansion in graphs which are associated with orbits of such groups. We review briefly the set up in \cite{BGS}. Let $\Gamma$ be a finitely generated group of affine linear motions of $\bbq^n$, that is transformations of the form $\phi:x \mapsto Ax+b$ with $A\in \GL_n(\bbq)$ and $b\in \bbq^n$. It will be convenient for us to realize $\Gamma$ as a subgroup of linear transformations of $\bbq^{n+1}$ by setting 
\[
\phi=\left[\begin{array}{cc}A&b^t\\0&1\end{array} \right].
\] 
Fix $v\in \bbq^n$ and let $\ocal=\Gamma v$ be the orbit of $v$ under $\Gamma$ in $\bbq^n$. Since $\Gamma$ is finitely generated the points of $\ocal$ have coordinates in the ring of  $S$-integer $\bbz_S$ (that is their denominators have all of their prime factors in the finite set $S$). In what follows, we will suppress the behavior of our points at these places in $S$ and we will even extend $S$ to a fixed finite set $S\rq{}$ when convenient. This is done for technical simplicity and an analysis of what happens at these places can probably be examined and controlled, but we will not do so here.

Denote by $\Zcl(\ocal)$ the Zariski closure of $\ocal$ in $\bba^n_{\bbq}$. Let $f\in \bbq[x_1,\ldots,x_n]$ and denote by $V(f)$ its zeros, we will assume henceforth that $\dim(V(f)\cap \Zcl(\ocal))<\dim(\Zcl(\ocal))$, i.e. $f$ is not constantly zero on any of the irreducible components of $\Zcl(\ocal)$.  We seek points $x\in \ocal$ such that $f(x)$ has at most a fixed number of prime factors outside of $S$ (or an enlarged $S\rq{}$). For $m\ge 1$ and $S\rq{}$ fixed (and finite) set
\begin{equation}\label{e:mSprime}
\ocal_{m,S\rq{}}:=\{x\in \ocal|\h f(x)\h\text{has at most $m$ prime factors outside $S\rq{}$}\}.
\end{equation} 
 Thus $\ocal_{1,S\rq{}}\subseteq \ocal_{2,S\rq{}}\subseteq \cdots$, $\ocal=\cup_{m=1}^{\infty} \ocal_{m,S\rq{}}$ and $\cup_{m=1}^{\infty} \Zcl(\ocal_{m,S\rq{}})\subseteq \Zcl(\ocal)$. We say that the pair $(\ocal,f)$ saturates if $\Zcl(\ocal)=\Zcl(\ocal_{r,S\rq{}})$ for some $r<\infty$. In words, this happens if there is a finite set $S\rq{}$ of primes and a finite number $r$ such that the set of points $x\in \ocal$ at which $f(x)$ has at most $r$ prime factors outside of $S\rq{}$ (that is to say at most $r$ prime factors as an $S\rq{}$-integer) is Zariski dense in $\Zcl(\ocal)$.

In \cite{BGS} many classical examples and applications of such saturation (or conjectured saturation) are given. Here we simply point to Brun\rq{}s original work. If $n=1$ and $\bbg=\Zcl(\Gamma)$ contains no tori, then $\ocal$ (if it is infinite) is essentially an arithmetic progression. In this case, the pair $(\ocal,f)$ saturates by Brun\rq{}s results, which assert that there are infinitely many $x\in \bbz$ (infinite is equivalent to Zariski density in $\bba^1$) such that $f(x)$ has at most a fixed number $r=r_f$ prime factors. In this case after Brun, much effort has gone into reducing the number $r$ (for example if $f(x)=x(x+2)$, then $r_f=2$ is equivalent to the twin prime conjecture and $r_f=3$  is known~\cite{Ch}). On the other hand in this one dimensional case if $\bbg=\Zcl(\Gamma)$ is a torus, then it is quite likely that $(\ocal,f)$ does not saturate for certain $f$\rq{}s. For example if $\Gamma=\{2^n|\h b\in \bbz\}$, $v=1$ and $f(x)=(x-1)(x-2)$, then standard heuristic probabilistic arguments (see for example \cite[Page 15]{HW} or \cite{BLMS} for a related conjecture and heuristic argument) suggest that the number of odd distinct prime factors of $(2^m-2)(2^m-1)$ tends to infinity as $m$ goes to infinity. That is $(\ocal,f)$ does not saturate.

This feature persists (see the Appendix) for any group which does not satisfy one of the following equivalent conditions for a group $\bbg=\Zcl(\Gamma)$.
\begin{enumerate}
\item The character group $X(\bbg^{\circ})$ of $\bbg^{\circ}$ is trivial, where $\bbg^{\circ}$ is the connected component of $\bbg$. 
\item No torus is a homomorphic image of $\bbg^{\circ}$.
\item $X(R(\bbg))=1$, where $R(\bbg)$ is the radical of $\bbg$.
\item $\bbg/R_u(\bbg)$ is a semisimple group, where $R_u(\bbg)$ is the unipotent radical of $\bbg$.
\item $\bbg\simeq\bbg_{ss}\ltimes \bbu$, where $\bbg_{ss}$ is a semisimple group and $\bbu$ is a unipotent group. 
\item The Levi factor of $\bbg$ is semisimple.
\end{enumerate}
If $\bbg$ satisfies the above properties, we call it {\it Levi-semisimple}.. We can now state our main result which is a proof of the {\it fundamental saturation theorem} that was conjectured in \cite{BGS}.  
\begin{thm}\label{t:main2}
Let $\Gamma$, $\ocal$ and $f$ be as above and assume that $\bbg=\Zcl(\Gamma)$ is Levi-semisimple, then $(\ocal,f)$ saturates. That is there are a positive integer $r$ and finite set of primes $S'$  such that $\Zcl(\ocal_{r,S\rq{}})=\Zcl(\ocal)$.
\end{thm}
\begin{remark}
\begin{enumerate}
\item The condition on $\bbg$ which is quite mild and easily checked in examples, is probably necessary for saturation (in particular it is needed in executing a Brun like sieve), when considering all pairs $(\ocal,f)$ for which $\Zcl(\Gamma)=\bbg$. We discuss the heuristics leading to this belief in the Appendix. (It is worth emphasizing however that we have no example of a pair $(\ocal,f)$ for which we can prove does not saturate!) These heuristics indicate what we expect is the case, that the condition on $\bbg$ in the theorem is the exact one that leads to saturation. 

\item The proof of the Theorem~\ref{t:main2} is effective in the sense that given a pair $(\ocal,f)$, there is an algorithm which will terminate with a value $r$ and the set $S\rq{}$. However without imposing strong conditions on $\Gamma$ (such as it being a lattice in the corresponding group $G_S$, as is done in \cite{NS}) the bounds for $r$ that would emerge from our proof would be absurdly large and very far from the minimal $r$ (called the saturation number in \cite{BGS}).
\end{enumerate}
\end{remark}

We outline the proof of Theorem~\ref{t:main2}. We start by pulling back $f$ to a regular function on $\bbg$ and reformulate Theorem~\ref{t:main2} to the following form. 

\begin{thm}\label{t:Main}
Let $\Gamma$ be a finitely generated subgroup of $\SL_n(\bbq)$. Let $\bbg$ be the Zariski closure 
of $\Gamma$ in $(\mathbb{SL}_n)_{\bbq}$ and $f\in \bbq[\bbg]$ which is not constantly zero on any of 
the irreducible components of $\bbg$. If $\bbg$ is Levi-semisimple, then there are a positive integer $r$ and a finite set $S$ of primes such that 
\begin{equation}\label{e:PrimeGamma}
\Gamma_{r,S}(f):=\{\gamma\in \Gamma|\h f(\gamma)\h\text{has at most $r$ prime factors outside $S$}\}
\end{equation}
is Zariski dense in $\bbg$.
\end{thm}

To prove Theorem~\ref{t:Main}, first we find a perfect normal subgroup $\bbh$ of $\bbg$ such that $\Gamma\cap\bbh$ is Zariski-dense in $\bbh$ and $\bbg/\bbh$ is a unipotent group.  Let $\pi$ denote the the projection map $\pi:\bbg\rightarrow \bbg/\bbh$. Since $\bbg/\bbh$ is a $\bbq$-unipotent group, there is a $\bbq$-section $\phi:\bbg/\bbh\rightarrow \bbg$ and, as a $\bbq$-variety, $\bbg$ can be identified with the product of $\bbh$ and $\bbu$ (see Section~\ref{s:General} for more details). Thus there are polynomials $p$ and $p_i\in \bbq[\bbu]$ and regular functions $f_i\in\bbq[\bbh]$ such that $\gcd(p_i)=1$ and
\begin{equation}\label{e:Product}
f(g)=p(\pi(g))\cdot \left[\sum_i p_i(\pi(g))\ f_i(\phi(\pi(g))^{-1}g)\right]=p(\pi(g))\cdot \left[\sum_i p_i(\pi(g))L_{\phi(\pi(g))}(f_i)(g)\right],
\end{equation}
where $L_g:\bbq[\bbg]\rightarrow\bbq[\bbg]$ is the left multiplication operator, i.e. $L_g(f)(g')=f(g^{-1}g')$.

In the second step, we prove the following stronger version of Theorem~\ref{t:Main} for a unipotent group to get a control on the values of $p$ and $p_i$'s. 

\begin{thm}~\label{t:unipotent}
Let $\bbu$ be a unipotent $\bbq$-group. Let $\Gamma$ be a finitely generated, Zariski dense subgroup of $\bbu(\bbq)$, and $p, p_1,\ldots,p_m\in\bbq[\bbu]$ such that $\gcd(p_i)=1$. Then there are a finite set $S$ of primes and a positive integer $r$ such that, 
\[
\{\gamma\in\Gamma\h|\h p(\gamma)\h\text{has at most $r$ prime factors in $\bbz_S$ and $\gcd(p_i(\gamma))$ is a unit in $\bbz_S$}\}
\]
is Zariski dense in $\bbu$.
\end{thm}

The major inputs in the proof of Theorem~\ref{t:unipotent} are Malcev theory of lattices in Nilpotent Lie groups and Brun's combinatorial sieve. 

Using Theorem~\ref{t:unipotent}, in order to show Theorem~\ref{t:Main}, one needs to prove its stronger form for perfect groups which also provides a uniform control on $r$ and $S$ for all the coprime linear combinations of a finite set of regular functions $f_i$'s. We get such a control in two steps. Before stating the precise formulation of our results, let us briefly recall parts of Nori's results from~\cite{Nor} and introduce a few notations. 

As we said earlier, $\Gamma\subseteq \SL_n(\bbz_{S_0})$ for some finite set of primes $S_0$. Let $\gcal$ be the Zariski-closure of $\Gamma$ in $(\mathbb{SL}_n)_{\bbz_{S_0}}$. It is worth mentioning that $\bbg$ is just the generic fiber of $\gcal$. If $\bbg$ is generated by its 1-parameter unipotent subgroups, then, by~\cite{Nor}, there is a finite set $S_0\subseteq S_{\Gamma}$ of primes  such that
\begin{enumerate}
\item  The projection map $\gcal \times \Spec(\bbz_{S_{\Gamma}})\rightarrow \Spec(\bbz_{S_{\Gamma}})$ is smooth.
\item All the fibers of the projection map $\gcal \times \Spec(\bbz_{S_{\Gamma}})\rightarrow \Spec(\bbz_{S_{\Gamma}})$ are geometrically irreducible and have the same dimension.
\item $\pi_p(\Gamma)=\gcal_p(\f_p)$ for any $p$ outside of $S_{\Gamma}$, where $\pi_p:\Gamma\rightarrow\SL_n(\f_p)$ is the homomorphism induced by the residue map $\pi_p:\bbz\rightarrow\f_p$ and $\gcal_p=\gcal\times \Spec(\f_p)$.
\item $\prod_{p\not\in S_{\Gamma}}\gcal(\bbz_p)$ is a subgroup of the closure of $\Gamma$ in $\prod_{p\not\in S_0}\SL_n(\bbz_p)$ (where $\bbz_p$ is the ring of $p$-adic integers).
\end{enumerate}

For a given $f\in \bbq[\bbg]$, there is a finite set $S_{\Gamma}\subseteq S$ of primes such that $f\in \bbz_S[\gcal]$ (we take the smallest such set). We say that $p$ is a ramified prime with respect to $\Gamma$ and $f$, if  $f(\gamma)\in p\bbz_{S}$ for any $\gamma\in\Gamma$. We denote the set of all the ramified primes with respect to $\Gamma$ and $f$ by $S_{\Gamma,f}$. 

We also note that $f\in \bbq[\bbg]$ can be lifted to a regular function $\tilde{f}$ on all the $n\times n$ matrices (we pick one of such lifts with smallest possible degree). 

\begin{thm}\label{t:Perfect}
In the above setting, if $\bbg$ is perfect and generated by its unipotent subgroups, then there is a positive integer $r$ depending on $\Gamma$, the degree of $\tilde{f}$, a lift of $f$ to $\bba^{N^2}$, and $\#S_{\Gamma,f}$ such that $\Gamma_{r,S_{\Gamma,f}}(f)$ is Zariski dense in $\bbg$. 
\end{thm}

To establish Theorem~\ref{t:Perfect}, we carefully follow the treatment given in the work of Bourgain, Gamburd and Sarnak~\cite{BGS} and combine it with a recent result of Salehi Golsefidy and Varj\'{u}~\cite{SV}\footnote{This work relies in part on a number of recent developments (\cite{H}, \cite{BG}, \cite{BGT},\cite{PS}, \cite{V})}. Theorem~\ref{t:Perfect} enables us to get a fixed $r$ that works for all the linear combinations of a given finite set of regular functions $f_i$'s, as soon as we have a uniform control on the set of associated ramified primes. In the second step, we get a uniform upper bound on the ramified primes with respect to $\Gamma$ and all the coprime linear combinations of $f_i$'s.

\begin{thm}\label{t:Perfect2}
In the above setting, let $\bbg$ be Zariski-connected and perfect. Then for any finite set of primes $S'$ and any given $f_1, \ldots, f_m\in \bbq[\bbg]$ which are linearly independent over $\bbq$, there are a positive integer $r$ and a finite set $S$ of primes such that $\Gamma_{r,S}(f_{\vbf,g})$ is Zariski dense in $\bbg$ 
 for any primitive integer vector $\vbf=(v_1,\ldots,v_m)$ and any $g\in \gcal(\bbz_{S'})$, where $f_{\vbf,g}=L_g(\sum_{i=1}^m v_i f_i)$.
\end{thm} 

Using Theorem~\ref{t:Perfect2}, we are able to finish the proof of Theorem~\ref{t:Main}.

In this paragraph, we fix a few notations that will be used in the rest of article. Let $\Pi$ be the set of all the primes. For any rational number $q$, let $\Pi(q)$ be the set of all the prime factors of $q$ (with a positive or negative power). For a Zariski-connected group $\bbg$, let $R(\bbg)$ (resp. $R_u(\bbg)$) be the radical (resp. the unipotent radical) of $\bbg$ and let $\bbg_{ss}:=\bbg/R(\bbg)$ be the semisimple factor of $\bbg$. If $\bbg$ is a Zariski-connected, Levi-semisimple group, then $\bbg\simeq \bbg_{ss}\ltimes R_u(\bbg)$. Let $\bbz^m_*$ be the set of all the primitive $m$-tuple of integers. For any affine scheme $X=\Spec(A)$ and a regular function $f$ on $X$, $V(f)$ denotes the closed subscheme of $X$ defined by $f$, i.e. $V(f):=\{\pfr\in \Spec(A)|\h f\in \pfr\}$.

\section{The unipotent case.}\label{s:unipotent}
In this section, we prove Theorem~\ref{t:unipotent}. We start with the abelian case. 
\begin{lem}~\label{l:SingleVariableSieve}
Let $P(x)\in\bbz[x]$; then there is a positive integer $r=r(\deg P)$ which depends only on $\deg P$ such that $P(n)$ has at most $r$ prime factors outside $S_P=\gcd_{m\in \bbz} P(m)$ for infinitely many integer $n$.
\end{lem}
\begin{proof}
This is a classical result of sieve theory~\cite{H??}.
\end{proof}
\begin{lem}~\label{l:BadPrimesPolynomial}
For a given $P(x)=\sum_{i=0}^m a_i x^i \in \bbz[x]$, one has 
\[
S_P\subseteq [1,\deg P]\cup \Pi(\gcd_i\h a_i). 
\]
\end{lem}
\begin{proof}
It is clear.
\end{proof}
\begin{lem}~\label{l:AvoidingPrimeFactors}
For $M\in \bbz$ and $P_1(x),\ldots,P_m(x)\in\bbz[x]$, there are integers $a$ and $b$ such that
\[
\bigcup_{i=1}^m\Pi(\gcd(P_i(aj+b),M))\subseteq \bigcup_{i=1}^k S_{P_i},
\]
for any integer $j$.
\end{lem}
\begin{proof}
For a given prime $p$ which is not in $\bigcup_{i=1}^k S_{P_i}$, by the definition, for some $i$ and $b_p$, $P_i(pj+b_p)$ is coprime to $p$ for any integer $j$. One can complete the argument by the Chinese remainder theorem.  
\end{proof}
\noindent
In the following lemma, we prove a stronger version of Theorem~\ref{t:unipotent} when $\bbu=\bbg_a^d$.
\begin{lem}~\label{l:MultivariableSieve}
Let $P, P_{ij}\in\bbz[x_1,\cdots,x_d]$ for $1\le i\le m$ and $1\le j\le m'$. Assume that, for any $1\le i\le m$, $\gcd_{j=1}^{m'} (P_{ij})=1$. Then there are a positive integer $r$, a finite set $S$ of primes, and a Zariski dense subset $X$ of $\bbz^d$ such that for any $\mathbf{x}\in X$,
\begin{enumerate}
\item $P(\mathbf{x})$ has at most $r$ prime factors outside $S$.
\item $\bigcup_{i=1}^m \Pi(\gcd_{j=1}^{m'}(P_{ij}(\mathbf{x})))\subseteq S.$
\end{enumerate}
\end{lem} 
\begin{proof}
We proceed by induction on $d$ the number of variables. For $d=1$ and any $i$, there are integer polynomials $Q_{ij}$ such that $\sum_{j=1}^{m'} Q_{ij}(x)P_{ij}(x)=m_i\in\bbz$. Hence for any $x\in\bbz$, $\gcd_{j=1}^{m'}(P_{ij}(x))$ divides $m_i$, and,  by  Lemma~\ref{l:SingleVariableSieve}, we are done.
\\

\noindent
 For the induction step, without loss of generality, by increasing $m$ if necessary, we may and will assume that $P_{ij}$'s are irreducible polynomials.  Viewing $P_{ij}$'s as polynomials on $x_d$, for any $i$, we can find polynomials $Q_{ij}(x_1,\ldots,x_d)$ such that $\sum_{j=1}^{m'} Q_{ij} P_{ij}=Q_i\in\bbz[x_1,\ldots,x_{d-1}]$. Let $P_{ij}=\sum_{l} H_{ij}^{(l)} x_d^l$, where $H_{ij}^{(l)}\in\bbz[x_1,\ldots,x_{d-1}]$. Since $P_{ij}$ is irreducible, either $P_{ij}$ is independent of $x_d$ or $\gcd_{l}(H_{ij}^{(l)})=1$.   We also write $P$ as a polynomial in $x_d$: 
 \[
 P(x_1,\ldots,x_d)=H(x_1,\ldots,x_{d-1}) \sum H_i(x_1,\ldots,x_{d-1}) x_d^i,
 \]
 where $\gcd_i  H_i =1$. Now let's apply induction hypothesis for the following polynomials:
 \begin{enumerate}
 \item Let $H(x_1,\cdots,x_{d-1})$ be  the new $P$.
 \item Let $\{H_i\}$ be one of the sequence of coprime polynomials.
 \item For a given $i$, either $\{P_{ij}\}_j$ if all of them are independent of $x_d$ or $\{H_{ij}^{(l)}\}_l$ where $P_{ij}$ depends on $x_d$.
 \end{enumerate}
\noindent 
 Because of the way we chose the sequences of polynomials, certainly, the g.c.d. of each sequence is 1, and we are allowed to use the induction hypothesis. So we get  a positive number $r$, a finite set $S$ of primes, and  a Zariski dense subset $X$ of $\bba^{d-1}$, such that for any ${\bf x} =(x_1,\cdots,x_{d-1})\in X,$
 
 \begin{enumerate}
 \item $H({\bf x})$ has at most $r$ prime factors in the ring of $S$-integers.
 \item $\Pi(\gcd_i(H_i({\bf x})))\subseteq S$.
 \item For a given $i$, either $\Pi(\gcd_j(P_{ij}({\bf x})))\subseteq S$ if all of $\{P_{ij}\}_j$ are independent of $x_d$ or $\Pi(\gcd_l(H_{ij}^{(l)}({\bf x})))\subseteq S$ where $P_{ij}$ depends on $x_d$.
 \end{enumerate}
 \noindent
 Let us fix ${\bf x}\in X$, and set $M=\prod_i Q_i({\bf x})$. For a given $i$, if  $\{P_{ij}\}_j$ are independent of $x_d$, for all $j$, we have  already got the condition on the g.c.d. of their value. So let us just focus on $i$'s for which there is $j$ such that $P_{ij}$ depends on $x_d$. For all such $i$ and $j$, consider the single variable polynomials $\pcal_{ij}(x_d)=\sum_k H_{ij}^{(k)}({\bf x}) x_d^k$. Let us also introduce $\pcal(x_d)=\sum H_i({\bf x}) x_d^i$. By Lemma~\ref{l:BadPrimesPolynomial} and Lemma~\ref{l:AvoidingPrimeFactors}, there is an arithmetic progression $ax+b$, such that for any $x\in\bbz$, 
 \[
 \Pi(\gcd(\pcal(ax+b),M))\cup\bigcup_{i,j}\pi(\gcd(\pcal_{ij}(ax+b),M))\subseteq [1,\deg p+\sum_{i,j}\deg p_{ij}]\cup S.
 \]
 \noindent
 Let $\widetilde{S}=\Pi\cap([1,\deg \pcal+\sum_{i,j}\deg \pcal_{ij}]\cup S)$. Thus, by this discussion and Lemma~\ref{l:BadPrimesPolynomial}, we have that for any integer $x$ and ${\bf x}\in X$,
 \[
 S_{\pcal(ax+b)}\cup \bigcup_i\Pi(\gcd_j(P_{ij}({\bf x}, ax+b)))\subseteq \widetilde{S}.
 \]
 Now classical sieve on $\pcal(ax+b)$ and induction hypothesis give us $\tilde{r}$ such that for any ${\bf x}\in X$, we would be able to find an infinite subset of integer numbers $V_{\bf x}$ with the following properties:
 \begin{enumerate}
 \item $P({\bf x},x_d)$ has at most $\tilde{r}$ prime factors in the ring of $\widetilde{S}$-integers, for any $x_d\in V_{\bf x}$.
 \item $\bigcup_i \Pi(\gcd_{j}(P_{ij}({\bf x},x_d)))\subseteq \widetilde{S},$ for any $x_d\in V_{\bf x}$. 
 \end{enumerate}
 \noindent
 Hence $\tilde{r}$, $\widetilde{S}$, and $\bigsqcup_{{\bf x}\in X} \{{\bf x}\}\times V_{\bf x}$ satisfy our claim. 
\end{proof}
\begin{lem}~\label{l:FinitelyGeneratedDiscrete}
A finitely generated subgroup of the group of unipotent upper-triangular rational matrices  ${\rm U}_n(\bbq)$ is discrete in ${\rm U}_n(\bbr)$.
\end{lem}
\begin{proof}
 Let $d_N=\diag(N^{n-1},N^{n-2},\cdots,1)$.
The claim is a clear consequence of the fact that $${\rm U}_n(\bbq)=\bigcup_{N\in \bbn} d_N^{-1}\cdot {\rm U}_n(\bbz)\cdot  d_N,$$ and if $N$ divides $M$,
\[
d_N^{-1}\cdot {\rm U}_n(\bbz)\cdot  d_N\subseteq d_M^{-1}\cdot {\rm U}_n(\bbz)\cdot  d_M.
\]
\end{proof}
\begin{lem}~\label{l:UnipotentZariskiDense}
 Let $\bbu$ be a unipotent $\bbq$-group. If $\Gamma$ is a finitely generated, Zariski dense subgroup of $\bbu(\bbq)$, then there is a lattice $\Lambda$ in $\mathfrak{u}=\Lie(\bbu)(\bbr)$ such that $\exp(\Lambda)$ is a subset of $\Gamma$.
\end{lem}
\begin{proof}
By Lie-Kolchin theorem, $\bbu$ can be embedded in ${\rm U}_n$, for some $n$, as a $\bbq$-group. Therefore, by Lemma~\ref{l:FinitelyGeneratedDiscrete}, $\Gamma$ is a closed subgroup of $U=\bbu(\bbr)$. Hence by~\cite[Theorem 2.12]{Ra}, it is a lattice in $U$. So $\Lambda_0$ the $\bbz$-span of $\log \Gamma$ is a lattice in $\mathfrak{u}$, and  $\widetilde{\Gamma}=\langle \exp(\Lambda) \rangle$ is a finite extension of $\Gamma$. In particular, for some $m$, $\exp(m \Lambda_0)$ is a subset of $\Gamma$, as we desired. 
 \end{proof}
\begin{proof}[Proof of Theorem~\ref{t:unipotent}]
Since $\bbu$ is a unipotent group, $\Lie(\bbu)$ can be identified with the underlying $\bbq$-variety of $\bbu$ via the exponential map $\exp:\Lie(\bbu)\rightarrow \bbu$. Via this identification, we can and will view $p$ and $p_i$'s as regular functions on $\Lie(\bbu)$, i.e. polynomials in $d=\dim \bbu$ variables. By Lemma~\ref{l:UnipotentZariskiDense}, we find a lattice $\Lambda$ of $\Lie(\bbu)(\bbr)$ such that $\exp(\Lambda)\subseteq \Gamma$. Since the logarithmic map is defined over $\bbq$, $\Lambda$ is a subgroup of $\Lie(\bbu)(\bbq)$. Hence we can identify $\Lie(\bbu)(\bbq)$ with $\bbq^d$ in a way that $\Lambda$ gets identified with $\bbz^d$. Now one can easily finish the proof using Lemma~\ref{l:MultivariableSieve}.
\end{proof}

\section{The perfect case I. }
The goal of this section is to prove Theorem~\ref{t:Perfect}. To do so, we essentially follow~\cite{BGS}. However we have to be extra careful as we need to understand how $r$ and $S$ depend on $f$. 

In this section, we will assume that the Zariski-closure $\Gamma$ is perfect and it is generated by its unipotent subgroups. It is equivalent to say that $\bbg\simeq \bbg_{ss}\ltimes R_u(\bbg)$ is perfect and $\bbg_{ss}$ is Zariski-connected and simply connected. 

\begin{prop}\label{p:Geometric}
In the above setting, let $f$ be a non-zero element of $\bbq[\bbg]$ and let $\tilde{f}$ be a lift of $f$ to $\bba^{n^2}_{\bbq}$. Then there is  a positive integer $M$ which depends only on the degree of $\tilde{f}$ such that $V(f)$ has at most $M$ geometric irreducible components.
\end{prop}
\begin{proof}
Let $g_i\in \bbq[\bba^{N^2}]$ be defining relations of $\bbg$. So $V(f)$ is isomorphic to 
\[
\Spec(\bbq[x_1,\ldots,x_{N^2}]/\langle \tilde{f}, g_1, \ldots, g_k\rangle).
\]
We also notice that the number of irreducible components of a subvariety of $\bba^{N^2}$ is the same as the number of irreducible components of its closure in $\bbp^{N^2}$. On the other hand, by general Bezout's theorem~\cite{Sc}, we have that 
\[
\sum_i\deg W_i\le \deg \overline{V(\tilde{f})}\cdot\prod_i \deg \overline{V(g_i)},
\]
where $W_i$ are the irreducible components of the projective closure of $V(f)$. This complete the proof.
\end{proof}
\begin{lem}\label{l:PointsEmpty}
Let $V$ be a closed subset of $\bba^n_{\f_p}$ defined over $\f_p$. If $\f$ is a non-trivial extension of $\f_p$ and $\Aut(\f)$ acts simply transitively on the geometric irreducible components of $V$, then
\[
\#V(\f_p)=O(p^{\dim V-1}),
\]
where the constant depends only on $n$, the geometric degree and the geometric dimension of $V$.
\end{lem}
\begin{proof}
By the assumption, $V=\bigcup_{\sigma\in\Aut(\f)}W^{\sigma}$ and $\dim(W\cap W^{\sigma})\le \dim V-1$ when $\sigma$ is not identity. It is clear that $V'=\bigcap_{\sigma\in \Aut(\f)} W^{\sigma}$ is also defined over $\f_p$. We claim that
\[
V(\f_p)=V'(\f_p).
\]
 To show the claim, we note that any irreducible component is also affine. Let $\xbf_{0}\in V(\f_p)$. Then $\xbf_0\in W^{\sigma}$ for some $\sigma\in \Aut(\f)$, i.e. $f(\xbf_0)=0$ for any $f\in \f[x_1,\ldots,x_n]$ which vanishes on $W^{\sigma}$. Let  $f=\sum_j \lambda_j f_j$ where $\{\lambda_j\}$ is an $\f_p$-basis of $\f$ and $f_j\in \f_p[x_1,\ldots,x_n]$. Thus $f_j(\xbf_0)\in \f_p$. Since $\lambda_j$'s are linearly independent over $\f_p$ and $\sum_j \lambda_j \f_j(\xbf_0)=0$, we have that $f_j(\xbf_0)=0$ for any $j$. Hence $\sigma'(f)(\xbf_0)=0$ for any $\sigma'\in \Aut(\f)$, which completes the proof of the claim. We can also control the degree of $V'$ by the degree and the dimension of $V$. Hence, by~\cite[Lemma 3.1]{FHM}, the Lemma follows.
\end{proof}
If $V$ is a closed subset of $\bba^n_{\bbq}$, by the definition there are polynomials $p_1,\ldots, p_k\in \bbq[x_1,\ldots,x_n]$ such that $V=\{\pfr \in\Spec \bbq[x_1,\ldots,x_n]|\h \langle p_1,\ldots,p_k\rangle\subseteq \pfr\}$. For a large enough prime $p$, we can look at $p_i$\rq{}s modulo $p$ and get a new variety $V_p$ over $\f_p$. Changing $p_i$\rq{}s to another set of defining relations only changes finitely many $V_p$. Hence for almost all $p$,  $V_p$ just depends on $V$. We will abuse notation and use $V(\f_p)$ instead of $V_p(\f_p)$ in the following statements.

\begin{prop}\label{p:PointsQIrreducible}
Let $V$ be a closed subset of $\bba^n_{\bbq}$ all of whose geometric irreducible components are defined over $k$ and  the group of automorphism $\Aut(k)$ of $k$ acts simply transitively on the geometric irreducible components of $V$. Then for almost all $p$,
\[
\#V(\f_p)=\begin{cases}
O(p^{\dim V-1})&\text{if $p$ does not completely split over $k$,}\\
\deg(k)\h p^{\dim V}+O(p^{\dim V-\frac{1}{2}})&\text{otherwise,}
\end{cases}
\]
and the implied constants depend only on $n$, the degree of $k$, the geometric degree and the geometric dimension of $V$.
\end{prop}
\begin{proof} Let $W$ be an irreducible component of $V$; then, by the assumption, we have 
$
V=\bigcup_{\sigma\in\Aut(k)} W^{\sigma},
$
$W^{\sigma}$'s are distinct irreducible components of $V$ and
\[
V_p=\bigcup_{\pfr| p}\bigcup_{\sigma\in\Aut(\f_{\pfr})} W_{\pfr}^{\sigma},
\]
where $W_{\pfr}$ is defined by defining relations of $W$ modulo $\pfr$ (it is well-defined for almost all $\pfr$). Let $V_{\pfr}=\bigcup_{\sigma\in\Aut(\f_{\pfr})} W_{\pfr}^{\sigma}$. By N\"{o}ether-Bertini's theorem, $W_{\pfr}$ is irreducible for almost all $\pfr$. In particular,
\[
\dim(V_{\pfr_1}\cap V_{\pfr_2})\le \dim V-1,
\]
for $\pfr_1\neq \pfr_2$ and the degree of this variety has an upper bound which only depends on the degree and the dimension of $V$. Hence, by~\cite[Lemma 3.1]{FHM}, 
\[
\#V_p(\f_p)=\sum_{\pfr|p} \#V_{\pfr}(\f_p)+O(p^{\dim V-1}),
\]
where the constant just depends on the degree and the dimension of $V$. By Lemma~\ref{l:PointsEmpty}, $\#V_{\pfr}(\f_p)=O(p^{\dim V-1})$ unless $p$ completely splits over $k$. If $p$ completely splits over $k$, $V_{\pfr}$ is an irreducible variety. Thus, by Lang-Weil~\cite{LW}, we have that 
\[
\#V_{\pfr}(\f_p)=p^{\dim V}+O(p^{\dim V-\frac{1}{2}}),
\]
where the constant depends on $n$ and the degree and the dimension of $V$, which completes the proof. 
\end{proof}
\begin{cor}\label{c:PointsArbitrary}
Let $V$ be a closed subset of $\bba^n_{\bbq}$. Then there are number fields $k_i$'s such that
\[
\#V(\f_p)=\left(\sum_{p\text{ completely splits}/k_i} \deg(k_i)\right)\h p^{\dim V}+O(p^{\dim V-\frac{1}{2}}),
\]
where the constant depends only on $n$, the geometric degree and the geometric dimension of $V$. Moreover 
$
\sum_i \deg k_i
$
is at most the number of geometric irreducible components of $V$.
\end{cor}
\begin{proof}
This is a direct corollary of Proposition~\ref{p:PointsQIrreducible} and \cite[Lemma 3.1]{FHM}.
\end{proof}
\begin{cor}\label{c:PointsV}
In the above setting, there exists a positive integer $M$ depending only on the degree of $\tilde{f}$ such that one can find number fields $k_i$'s where $\sum_i\deg k_i\le M$ and such that for almost all $p$,
\[
\#V(f)(\f_p)=\left(\sum_{p\text{ completely splits}/k_i} \deg(k_i)\right)\h p^{\dim \bbg-1}+O(p^{\dim \bbg-\frac{3}{2}}).
\]
\end{cor}
\begin{proof}
This is a consequence of Proposition~\ref{p:Geometric} and Corollary~\ref{c:PointsArbitrary}.
\end{proof}
\noindent
Since we assumed that $\bbg$ is perfect and Zariski-connected, we can find a free Zariski-dense subgroup of $\Gamma$ (e.g. see~\cite{SV}). So without loss of generality, we can and will assume that $\Gamma$ is a free group.

Let us also recall that since $\Gamma$ is finitely generated, it is a subgroup of $\SL_n(\bbz_{S_0})$. Its Zariski-closure in $(\mathbb{SL}_n)_{\bbz_{S_0}}$ is denoted by $\gcal$. As we said in the introduction, whenever $\bbg$ is generated by its unipotent subgroups, the closure of $\Gamma$ in $\prod_{p\not\in S_{\Gamma}} \SL_n(\bbz_p)$ is equal to $\prod_{p\not\in S_{\Gamma}} \gcal(\bbz_p)$.

Let $\pi_d:\Gamma \rightarrow \SL_n(\bbz_{S_{\Gamma,f}}/d\bbz_{S_{\Gamma,f}})$ be the homomorphism induced by the quotient (residue) map $\pi_d:\bbz_{S_{\Gamma,f}}\rightarrow \bbz_{S_{\Gamma,f}}/d\bbz_{S_{\Gamma,f}}$ for any $d$ which is not a unit in the ring of $S_{\Gamma,f}$-integers. If $d$ is a unit in the ring of $S_{\Gamma,f}$ integers, we set $\pi_d$ to be the trivial homomorphism. Let $\Gamma(d):=\ker(\pi_d)$. By the definition it is clear that if $\pi_d(\gamma_1)=\pi_d(\gamma_2)$, then $\pi_d(f(\gamma_1))=\pi_d(f(\gamma_2))$.
\begin{lem}\label{l:Multiplicative}
Let $N_f(d)=\#\{\pi_d(\gamma)\h s.t.\h \pi_d(f(\gamma))=0\}$. Then $N_{f}(d)$ is a multiplicative function for square free integer numbers $d$.
\end{lem}
\begin{proof}
For any square-free integer $d$, let $d_{\Gamma,f}=\prod_{p|d, p\not\in S_{\Gamma,f}} p$. We notice that $\pi_d(\gamma)=\pi_{d_{\Gamma,f}}(\gamma)$ and $\bbz_{S_{\Gamma,f}}/d\bbz_{S_{\Gamma,f}}$ is isomorphic to $\bbz/d_{\Gamma,f}\bbz$. Thus $N_f(d)=N_f(d_{\Gamma,f})$.

 On the other hand, we know that the closure of $\Gamma$ in $\prod_{p\not\in S_{\Gamma}} \SL_n(\bbz_p)$ is equal to $\prod_{p\not\in S_{\Gamma}} \gcal(\bbz_p)$. Hence
\begin{equation}\label{e:FiniteQuotient}
\Gamma/\Gamma(d)\simeq \prod_{p|d_{\Gamma}} \Gamma/\Gamma(p)\simeq \prod_{p|d_{\Gamma}} \gcal_p(\f_p),
\end{equation}
where $d_{\Gamma}=\prod_{p|d, p\not\in S_{\Gamma}} p$. Therefore $N_f(d)=\prod_{p|d_{\Gamma,f}} \#V(f)(\f_p)$, which completes the proof of the Lemma. 
\end{proof}
\noindent
In order to prove Theorem~\ref{t:Perfect}, we use the combinatorial sieve formulated in~\cite{BGS}. However we need to adjust some of the definitions before we could proceed as we are working with rational numbers instead of integers. Let 
$f_{\Gamma}:\Gamma\rightarrow \bbz^+$ be the following map
\[
f_{\Gamma}(\gamma):=\prod_{p\not\in S_{\Gamma,f}}|f(\gamma)|_p^{-1}=|f(\gamma)|\prod_{p\in S_{\Gamma,f}}|f(\gamma)|_p,
\]
where $|\cdot|$ is the usual absolute value and $|\cdot|_p$ is the $p$-adic norm. We notice that $f_{\Gamma}(\gamma)\bbz_{S_{\Gamma,f}}=f(\gamma)\bbz_{S_{\Gamma,f}}$. In particular, $\pi_d(f(\gamma))=0$ if and only if $\pi_d(f_{\Gamma}(\gamma))=0$.

 Let $\Gamma$ be freely generated by $\Omega$, $d_{\Omega^{\pm 1}}(\cdot,\cdot)=d(\cdot,\cdot)$ the relative word metric, $l(\gamma)=d(I,\gamma)$ and
\[
a_n(L)=\#\{\gamma\in \Gamma|\h l(\gamma)\le L, f_{\Gamma}(\gamma)=n\}.
\]
Let $\|\gamma\|_{S_{\Gamma,f}}=\max\{\|\gamma\|, \|\gamma\|_p|\h p\in S_{\Gamma,f}\}$, where $\|\gamma\|$ (resp. $\|\gamma\|_p$) is the operating norm of $\gamma$ on $\bbr^n$ (resp. $\bbq_p^n$). It is clear that, if $l(\gamma)\le L$, then $\|\gamma\|_{S_{\Gamma,f}}$ is at most $C^{L}$, where
$C=\max\{\|\gamma\|_{S_{\Gamma,f}}|\h \gamma\in\Omega\}.$ Hence 
\begin{equation}\label{e:SNorm}
f_{\Gamma}(\gamma)\le C_f\cdot C^{\deg\tilde{f}(\#S_{\Gamma,f}+1)\cdot L},
\end{equation}
if $l(\gamma)\le L$. We also observe that if $a_n(L)\neq 0$, then $n$ has no prime factor in $S_{\Gamma,f}$.

Following the same computation as in~\cite[Pages 18-20]{BGS} and using the main result of Salehi Golsefidy and Varj\'{u}~\cite{SV}, for any square free $d$, we have that
\begin{equation}\label{e:moddSum}
\sum_{d|n} a_n(L)=\beta(d)X+ r(d,\{a_i\}),
\end{equation}
where $X=\sum_n a_n(L)$, $|r(d,\{a_i\})|\ll N_{f}(d) X^{\tau}$, $\tau<1$ is independent of the choice of the regular function $f$, and, if $d$ has no prime factor in $S_{\Gamma,f}$, then
\begin{equation}\label{e:Density}
\beta(d)=\frac{N_{f}(d)}{\#\Gamma/\Gamma(d)}.
\end{equation}
If $d$ has a prime factor in $S_{\Gamma,f}$, then $\beta(d)=0$.
\begin{lem}\label{l:DensityMultiplicative}
In the above setting, $\beta$ is multiplicative on the square-free numbers $d$. Moreover there is $c_1$ a positive real number (which may also depend on $f$) such that $\beta(p)\le 1-\frac{1}{c_1}$.
\end{lem}
\begin{proof}
By Equations (\ref{e:FiniteQuotient}) and (\ref{e:Density}), and the proof of Lemma~\ref{l:Multiplicative}, we have that
\[
\beta(d)=\prod_{p|d} \frac{\# V(f)(\f_p)}{\#\gcal_p(\f_p)}
\]
if $d$ does not have a prime factor in $S_{\Gamma,f}$ and $\beta(d)=0$ if $d$ has a prime factor in $S_{\Gamma,f}$. Hence $\beta$ is a multiplicative function for square-free integers and, for almost all $p$, by Corollary~\ref{c:PointsV}, we have that
$
\beta(p)\le M/p,
$
where $M$ just depends on the degree of $\tilde{f}$. We also notice that $\beta(p)<1$ for any $p$ from which the Lemma follows.
\end{proof}
\begin{lem}\label{l:LevelDistribution}
In the above setting, for any positive integer number $D$ and any positive number $\vare$, we have
\[
\sum_{d\le D} |r(d,\{a_i\})|\ll X^{\tau} D^{\dim\bbg+\vare},
\]
where the constant only depends on $\vare$ and the degree of $\tilde{f}$.
\end{lem}
\begin{proof}
By the above discussion, we know that $|r(d,\{a_i\})|\ll N_{f}(d) X^{\tau}$. On the other hand, 
\[
N_{f}(d)=\prod_{p|d_{\Gamma,f}} N_{f}(p)\le \prod_{p|d_{\Gamma,f}}C p^{\dim \bbg-1},
\]
where $C$ only depends on the degree of $\tilde{f}$. Thus
\[
N_{f}(d)\le C' d^{\dim\bbg-1+\vare}, 
\]
where $C'$ only depends on $\vare$ and the degree of $\tilde{f}$.  Hence 
\[
\sum_{d\le D} |r(d,\{a_i\})|\ll X^{\tau} D^{\dim\bbg+\vare},
\]
as we wished.
\end{proof}
\begin{lem}\label{l:SieveDim}
In the above setting, there are constants $T$, $c_f$ and $t_f$ such that 
\begin{enumerate}
\item $T$ depends on the degree of $\tilde{f}$.
\item $t_f\le T$,
\item $|\sum_{w\le p\le z} \beta(p)\log p- t_{f} \log \frac{z}{w}|\le c_{f},$ for $\max S_{\Gamma,f}=p_0<w\le z$.
\end{enumerate}
\end{lem}
\begin{proof}
By Equation~(\ref{e:Density}) and Corollary~\ref{c:PointsV}, 
\[
\begin{array}{rl}
\sum_{p_0\le p\le z} \beta(p) \log p=&\sum_{p_0\le p\le z} \left(\sum_{p\text{ completely splits}/k_i} \deg(k_i)\right) \frac{\log p}{p}+O(1)\\
&\\
=&\sum_i \deg(k_i)\sum_{p_0\le p\le z\& c.s./k_i} \frac{\log p}{p}+O(1).
\end{array}
\]
where the constant and $k_i$'s depend on $f$ and $\sum_i \deg(k_i)$ has an upper bound which only depends on $\deg \tilde{f}$.
By Chebotarev's density  theorem, we have
\[
\sum_{p_0\le p\le z\& c.s./k_i} \frac{\log p}{p}=\frac{1}{\deg k_i} \log z+O(1).
\]
Hence
\[
\sum_{p_0\le p\le z} \beta(p) \log p=\sum_i \log z+ O(1)=t_{f} \log z+O(1),
\]
where $t_f\le T$ for some $T$ which depends only on the degree of $\tilde{f}$. We also notice that the implied constants might depend on $f$ but they do not depend on $w$ and $z$, which finishes the proof.
\end{proof}
\begin{prop}\label{p:NumberAlmostPrime}
In the above setting, there is a positive number $T$ depending on the degree of $\tilde{f}$ such that for $z=X^{(1-\tau)/9T(\dim\bbg+1)}$ and large enough $L$ we have
\[
\frac{X}{(\log X)^{t_{f}}}\ll \sum_{\Pi(n)\cap [1,z]=\varnothing} a_n(L) \ll \frac{X}{(\log X)^{t_{f}}},
\]
where $t_{f}\le T$ and the implied constants depend on $f$ and $\Gamma$.
\end{prop}
\begin{proof}
This is a direct consequence of the formulation of a combinatorial sieve given in~\cite{BGS}, Lemma~\ref{l:DensityMultiplicative}, Lemma~\ref{l:LevelDistribution} and Lemma~\ref{l:SieveDim}.
\end{proof}
\begin{cor}\label{c:AlmostPrime}
In the above setting, there are constants $r$ and $T$ such that
\begin{enumerate}
\item $r$ only depends on $\deg \tilde{f}$, $\#S_{\Gamma,f}$ and $\Gamma$.
\item $T$ only depends on $\deg \tilde{f}$.
\item For large enough $L$ (depending on $f$), we have
\[
\frac{X}{(\log X)^{T}}\ll \#\{\gamma\in\Gamma|\h l(\gamma)\le L,\h f_{\Gamma}(\gamma)\h \text{has at most $r$ prime factors}\}.
\]
\end{enumerate}
\end{cor}
\begin{proof}
 If $\gamma$ contributes to the sum in Proposition~\ref{p:NumberAlmostPrime}, then any prime factor of $f_{\Gamma}(\gamma)$ is larger than $X^{(1-\tau)/9T(\dim\bbg+1)}$. On the other hand, by (\ref{e:SNorm}), we have that $f_{\Gamma}(\gamma)\le C_f\cdot C^{\deg\tilde{f}(\#S_{\Gamma,f}+1)\cdot L}$. Thus the number of such factors is at most
\[
\frac{9T(\log C_f+L(\#S_{\Gamma,f}+1) \deg \tilde{f}\cdot \log M_0)(\dim\bbg+1)}{(1-\tau)(L+1)\log\#\Omega}.
\]
So, for large enough $L$, the number of prime factors is at most
\[
r=\left\lfloor\frac{9(\#S_{\Gamma,f}+1) \deg\tilde{f}\cdot T\cdot(\dim\bbg+1)\cdot \log M_0 }{(1-\tau)\log\#\Omega}\right\rfloor+1,
\]
which finishes the proof of the Corollary.
\end{proof}
\begin{proof}[Proof of Theorem~\ref{t:Perfect}]
This is now a direct consequence of \cite[Proposition 3.2]{BGS} and Corollary~\ref{c:AlmostPrime}.
\end{proof}
\section{The perfect case II.}
In this section, we prove Theorem~\ref{t:Perfect2}. We assume that the Zariski-closure of $\Gamma$ is perfect and Zariski-connected. 

\begin{lem}\label{l:ReductionSC}It is enough to prove Theorem~\ref{t:Perfect2} when the semisimple part of $\bbg$ is simply connected.
\end{lem}
\begin{proof}
Since $\bbg$ is equal to its commutator subgroup, its Levi component is semisimple and, as it is also Zariski-connected, $\bbg\simeq\bbg_{ss}\ltimes R_u(\bbg)$ as $\bbq$-groups (see~\cite{M} or \cite[Theorem 2.3]{PR}). Let $\tbbg_{ss}$ be the simply connected covering of $\bbg_{ss}$ and $\tbbg=\tbbg_{ss}\ltimes R_u(\bbg)$. Thus we have the following short exact sequence
\[
1\rightarrow \mu \rightarrow \tbbg \xrightarrow{\iota} \bbg\rightarrow 1,
\]
where $\mu$ is the center of $\tbbg$. Let $\tilde{\Gamma}=\iota^{-1}(\Gamma)$ and $\Lambda=\tilde{\Gamma}\cap\tbbg(\bbq)$. One has the following long exact sequence
\[
\mu(\bbq)\rightarrow \tbbg(\bbq) \xrightarrow{\iota} \bbg(\bbq)\rightarrow H^1(\bbq,\mu).
\]
Thus $\Gamma/\iota(\Lambda)$ is a finitely generated, torsion, abelian group, and so it is finite. As $\mu$ is also finite, $\tilde{\Gamma}/\Lambda$ is also finite. Therefore $\Lambda$ and $\iota(\Lambda)$ are Zariski-dense in $\tbbg$ and $\bbg$, respectively, as $\tbbg$ and $\bbg$ are Zariski-connected. Moreover $\Lambda$ is finitely generated as $\Gamma$ is finitely generated and  $\Gamma/\iota(\Lambda)$ and $\mu$ are finite. 

Now let us take a $\bbq$-embedding of $\tbbg$ in $\mathbb{SL}_{N'}$ for some $N'$. It is clear that we can find a finite set $\tilde{S}$ of primes such that
\begin{enumerate}
\item $\Lambda\subseteq \SL_{N'}(\bbz_{\tilde{S}})$ and  $\Gamma\subseteq \SL_N(\bbz_{\tilde{S}})$.
\item $\iota$ can be extended to a map from $\tilde{\gcal}$ to $\gcal\times \Spec(\bbz_{\tilde{S}})$, where $\tilde{\gcal}$ is the Zariski-closure of $\Lambda$ in $(\mathbb{SL}_{N'})_{\bbz_{\tilde{S}}}$.
\end{enumerate}

Now let $f_i^*\in \bbq[\tbbg]$ be the pull back of $f_i$. Clearly $f_i^*$'s are $\bbq$-linearly independent. So if Theorem~\ref{t:Perfect2} holds for $\tilde{\bbg}$, then there are $r$ and $S$ such that $\Lambda_{r,S}(f^*_{{\bf v},1})$ is Zariski-dense for any ${\bf v}\in \bbz_*^m$. By the definition, this means that $\iota(\Lambda)_{r,S}(f_{{\bf v},1})$ is Zariski-dense. By the above discussion, we have that $\Gamma_{r,S}(f_{{\bf v},1})$ is Zariski-dense for any ${\bf v}\in \bbz_*^m$.

Now we shall assume that $f\in\bbz_S[\gcal]$ is not constant and $\iota(\Lambda)_{r,S}(f)$ is Zariski-dense for a positive integer $r$. We would like to show that after enlarging $S$, if necessary, we have that $\iota(\Lambda)_{r,S}(L_g(f))$ is also Zariski-dense for any $g\in\gcal(\bbz_{S\rq{}})$. Without loss of generality, let us assume that $S$ contains both $S\rq{}$ and $\tilde{S}$. Now let $f^*_g\in\bbz_{S}[\tilde{\gcal}]$ be the pull back of $L_g(f)$. By a similar argument to the above it is enough to show after enlarging $S$ we have that $\Lambda_{r,S}(f^*_g)$ is Zariski-dense in $\tbbg$ for any $g$. To get such a result it is enough, by Theorem~\ref{t:Perfect}, to get a uniform bound on the degree of lifts of $f^*_g$ and a uniform upper bound for their sets of ramified primes. The claim on the degree of these functions is clear. Now let $p$ be a ramified prime of $f^*_g$, this means that for any $\lambda\in \Lambda$ we have that $f^*_g(\lambda)\in p\bbz_{S}$. Hence $\pi_p(f(g^{-1}\iota(\lambda)))=0$. Thus by \cite{Nor} we have that a coset of $\iota(\tgcal_p(\f_p))$ is a subset of $V(f)(\f_p)$. This implies that 
\begin{equation}\label{e:LowerBoundV}
\#V(f)(\f_p)>\!\!> p^{\dim \bbg},
\end{equation}
 where the implied constant just depends on $\bbg$. By Corollary~\ref{c:PointsV}, (\ref{e:LowerBoundV}) cannot hold for large enough $p$ unless it is a ramified prime of $f$. This completes the proof of Lemma 24.
\end{proof}
\noindent
For the rest of this section, by Lemma~\ref{l:ReductionSC}, we can and will assume that the semisimple part of $\bbg$ is simply connected. Let us continue with a few elementary lemmas in commutative algebra.
\begin{lem}\label{l:Zsfree}
Let $A$ be a finitely generated integral domain of characteristic zero. Then there exists a finite set $S$ of primes such that $A_S=A\otimes_{\bbz} \bbz_S$ is a free $\bbz_S$-module. Moreover there are $\bbq$-algebraically independent elements  $x_1,\ldots, x_d$ in $A$, such that $A_S$ is a finitely generated $\bbz_S[x_1,\ldots,x_d]$-module.
\end{lem}
\begin{proof}
By the assumptions, $A_{\bbq}$ is a finitely generated $\bbq$-integral domain. By N\"{o}ether normalization lemma, $A_{\bbq}$ is an integral, and so finite, extension of a polynomial algebra $B_{\bbq}=\bbq[x_1,\ldots,x_d]$. Now one can easily find a finite set of prime numbers $S$, such that $A_S$ is a finite extension of $B_S=\bbz_S[x_1,\ldots,x_d]$, i.e. it is a finitely generated $B_S$-module. Let $L$ ($K$, resp.) be the field of fractions of $A_S$ ($B_S$, resp.). Since $B_S$ is integrally closed and $L/K$ is a separable extension, there is $\{v_1,\ldots,v_n\}$ a $K$-basis of $L$ such that 
\[
A_S\subseteq B_Sv_1\oplus B_Sv_2\oplus\cdots\oplus B_Sv_n,
\]
(see~\cite[Proposition 5.17]{AM}). Hence $A_S$ is a $\bbz_S$-submodule of a free module, and so it is a free $\bbz_S$-module.
\end{proof}
\begin{cor}\label{c:IndependentModp}
Let $A$ be as above. If  $a_1,\ldots,a_m$ are $\bbq$-linearly independent elements of $A$, then there is a finite set $S$ of primes such that $a_i({\rm mod}\h p)$'s are linearly independent over $\f_p$, for any $p\in\Pi\setminus S$.
\end{cor}
\begin{proof}This is a direct consequence of Lemma~\ref{l:Zsfree}.
\end{proof}
\begin{definition}\label{d:Height}
For any $P(T)=\sum_i c_i T^i$ in the ring of polynomials with coefficients in $\bbq[x_1,\ldots,x_d]$, we define the height of $P$ to be
\[
H(P)=\max_i \{\deg c_i\}.
\]
\end{definition}
\begin{lem}\label{l:DegC}
Let $A_S$ be a finitely generated $\bbz_S$-integral domain which is a finite extension of a polynomial ring $B_S=\bbz_S[x_1,\ldots,x_d]$. If $a_1,\ldots,a_m$ are $\bbq$-linearly independent elements of $A$, then there exists $D>0$, depending on $a_i$'s, such that any integer combination of $a_i$'s satisfies a monic polynomial over $B_S$ whose height is {\rm at most} $D$.
\end{lem}
\begin{proof} We prove the lemma for $m=2$ and the general case can be deduced by induction. Let $P_{\alpha}$ and $P_{\beta}$ be monic polynomials with coefficients in $B_S$ which are satisfied by $\alpha=a_1$ and $\beta=a_2$, respectively. It is clear that, for any integer $n$, there is a monic polynomial $\mathcal{Q}(T)\in B_S[T]$ such that $\mathcal{Q}(nT)=n^{\deg P_{\alpha}} P_{\alpha}(T)$. In particular, $n\alpha$ satisfies a monic polynomial in $B_S[T]$ with height at most equal to $H(P_{\alpha})$ and degree at most $\deg(P_{\alpha})$. Thus it is enough to show that $\alpha+\beta$ satisfies a monic polynomial over $B_S$ whose height is bounded by a function of $H(P_{\alpha})$, $H(P_{\beta})$, $\deg(P_{\alpha})$ and $\deg(P_{\beta})$.
\\

\noindent
Let $x-\alpha^{(1)},\ldots,x-\alpha^{(n_1)}$ and $x-\beta^{(1)},\ldots,x-\beta^{(n_2)}$ be the linear factors of $P_{\alpha}(x)$ and $P_{\beta}(x)$, respectively, in an extension of the field of fractions of $A$. So $\alpha+\beta$ satisfies 
\[
P_{\alpha+\beta}(x)=\prod_{i=1}^{n_1}\prod_{j=1}^{n_2}(x-\alpha^{(i)}-\beta^{(j)}).
\]
On the other hand, consider the $n_1+n_2+1$ variable polynomial 
\[
P(T,\alpha_1,\ldots,\alpha_{n_1},\beta_1,\ldots,\beta_{n_2})=\prod_{i=1}^{n_1}\prod_{j=1}^{n_2}(T-\alpha_i-\beta_j).
\]
Since $P$ is invariant under any permutation of $\alpha_i$'s, there are linearly independent symmetric polynomials $\scal_n$'s in $\alpha_i$'s and polynomials $\mathcal{Q}_n$ in $T$ and $\beta_j$'s such that 
\[
P(T,\alpha_i,\beta_j)=\sum_n \mathcal{Q}_n(T,\beta_j) \scal_n(\alpha_i)\h\&\h \deg \mathcal{Q}_n+\deg \scal_n \le n_1n_2.
\]
As $P$ is also invariant under any permutation of $\beta_j$'s, we have 
\[
P(T,\alpha_i,\beta_j)=\sum_{n,l} \scal_n(\alpha_i) \scal'_{nl}(\beta_j) \mathcal{Q}_{nl}(T)\h\&\h \deg \scal_k+\deg \scal'_{nl} \le n_1n_2,\]
where $\scal'_{nl}$'s are symmetric polynomials in $\beta_j$'s. Thus 
\[
P_{\alpha+\beta}(x)=\sum_{n,l}\scal_n(\alpha^{(i)})\scal_{nl}'(\beta^{(j)})\mathcal{Q}_{nl}(x).
\]
 Moreover $\scal_n(\alpha^{(i)}),\scal_{nl}'(\beta^{(j)})\in \bbz[x_1,\ldots,x_d]$  and 
\[
\deg(\scal_n(\alpha^{(i)})\scal_{nl}'(\beta^{(j)}))\le D=D(H(P_{\alpha}),H(P_{\beta}),\deg(P_{\alpha}),\deg(P_{\beta})),
\]
which finishes the proof.
\end{proof}
\begin{prop}\label{p:Arithmetic}
Let $\bbg={\rm Zcl}(\Gamma)$ be a Zariski-connected perfect group such that its semisimple factor is simply-connected. Let $f\in \bbq[\bbg]$ be a non-zero function and let $S\rq{}$ be a finite set of primes. Then there is a finite set $S$ of primes such that
\[
\bigcup_{g\in\gcal(\bbz_{S\rq{}})} \bigcup_{\vbf\in\bbz^m_*} S_{\Gamma,f_{\vbf,g}}\subseteq S.
\]
\end{prop}
\begin{proof}
Since $\bbg_{ss}$ is simply connected, by Nori's theorem~\cite{Nor}, the closure of $\Gamma$ in $\prod_{p\in\Pi\setminus S_{\Gamma}} \SL_N(\bbz_p)$ is equal to $\prod_{p\in\Pi\setminus S_{\Gamma}} \gcal(\bbz_p)$. In particular, $\#\pi_p(\Gamma)=\#\gcal_p(\f_p)=p^d+O(p^{d-\frac{1}{2}})$, for any $p\in \Pi\setminus S_{\Gamma}$, where $d=\dim \bbg$.
\\

\noindent
We also notice that by a similar argument as in Lemma~\ref{l:ReductionSC}, a large enough $p$ is a ramified prime of $f_{{\bf v},g}$ if and only if it is a ramified prime of $f_{{\bf v},1}$.  So it is enough to prove the proposition only for $f_{\vbf}=f_{\vbf,1}$'s.

On the other hand, by Corollary~\ref{c:IndependentModp}, there is a finite set $S_1$ of primes  such that $f_i$'s modulo $p$ are linearly independent over $\f_p$. Also, by Lemma~\ref{l:Zsfree}, there is a finite set $S_2$ of primes such that $\bbz_{S_2}[\gcal]$ is a finite extension of a polynomial ring over $\bbz_{S_3}$. Let $p\in S_{\Gamma,\sum_{i=1}^mv_if_i}\setminus (S_{\Gamma}\cup S_1\cup S_2)$, where $\vbf\in \bbz^m_{*}$. So $\phi(\sum_{i=1}^mv_if_i)=0$, for any homomorphism $\phi:\bbz[\gcal]\rightarrow\f_p$. Any such homomorphism can be extended to a homomorphism from $\bbz[\gcal]\otimes\bbz_{S_3}$ to $\f_p$. Let $A=\bbz[\gcal]$. We have that $A_{S_3}$ is a finite extension of $B_{S_3}=\bbz_{S_3}[x_1,\ldots,x_d]$. Hence there is a positive number $D_1$ (independent of $\vbf$) such that at most $D_1$ homomorphisms $\phi:A_{S_3}\rightarrow \f_p$ have the same restriction on $B_{S_3}$. In particular, 
\begin{equation}\label{e:InequalityHom}
\#\{\phi':B_{S_3}\rightarrow \f_p:\h\exists\phi\in \Hom(A_{S_3},\f_p)\h{\rm s.t.}\h\phi'=\phi|_{B_{S_3}}\}\ge \frac{1}{2D_1}p^d.
\end{equation}
On the other hand, by Lemma~\ref{l:DegC}, there is  a positive number $D_2$ depending only on $f_i$'s such that $f_{\vbf}=\sum_i v_if_i$ satisfies an equation of degree at most $D_2$
\begin{equation}\label{e:Equation}
\sum_i c^{(\vbf)}_i f_{\vbf}^i=0,
\end{equation}
where $c^{(\vbf)}_i\in B_{S_3}$ and $\deg c^{(\vbf)}_i\le D_2$. Thus, for any $\phi\in \Hom(A_{S_3},\f_p)$, we have $\phi(f_{\vbf})=0$ as $p$ is in $S_{\Gamma,f_{\vbf}}\setminus (S_1\cup S_2\cup S_{\Gamma})$. Hence, by Equation (\ref{e:Equation}), $\phi(c_i^{(\vbf)})=0$ for some $i$. Therefore, by (\ref{e:InequalityHom}), we have
\begin{equation}\label{e:InequalityPoints}
\#V(\prod_i c^{(\vbf)}_i)(\f_p)=\#\{\phi'\in Hom(B_{S_3},\f_p):\h\phi'(\prod_i c^{(\vbf)}_i)=0\}\ge \frac{1}{2D_1}p^d.
\end{equation}
Notice that, since $p\not\in S_2$, $f_{\vbf}({\rm mod}\h p)$ is not zero and neither is $\prod_i c^{(\vbf)}_ i({\rm mod}\h p)$. Thus (see~\cite{Sch})
\begin{equation}\label{e:HyperPoint}
\#V(\prod_i c_i^{(\vbf)})(\f_p)\le \deg(\prod_i c_i^{(\vbf)}) p^{d-1}\le D_2^{D_2} p^{d-1}.
\end{equation}
Proposition~\ref{p:Arithmetic} now follows from (\ref{e:InequalityPoints}) and (\ref{e:HyperPoint}).
\end{proof}
\begin{lem}\label{l:UpperBoundDegree}
In the above setting, for any $g\in\SL_N(\bbq)$ and $\vbf\in\bbz^m$,
\[
\deg L_g(\sum_{i=1}^m v_i\tilde{f}_i)\le \max_{i} \deg\tilde{f}_i.
\] 
\end{lem}
\begin{proof}It is clear.
\end{proof}
\begin{proof}[Proof of Theorem~\ref{t:Perfect2}]
It is a direct consequence of Theorem~\ref{t:Perfect}, Proposition~\ref{p:Arithmetic}, and Lemma~\ref{l:UpperBoundDegree}.
\end{proof}

\section{The general case.}\label{s:General}
In this section, we complete the proof of Theorem~\ref{t:Main}. To do so, first we reduce it to the case of Zariski connected groups, and then carefully combine the perfect case with the unipotent case. 
\begin{lem}~\label{l:ReductionZariskiConnected}
If Theorem~\ref{t:Main} holds when $\bbg$ is a Zariski connected group, then it holds in general.
\end{lem} 
\begin{proof}
Let $\Gamma^{\circ}=\Gamma\cap\bbg^{\circ}$, where $\bbg^{\circ}$ is the connected component of $\bbg$ containing the identity element. Since $\Gamma$ is Zariski dense in $\bbg$, $\bbg/\bbg^{\circ}\simeq \Gamma/\Gamma^{\circ}$. Let \{$\gamma_i\}$ be a set of coset representatives of $\bbg^{\circ}$ in $\bbg$ chosen from $\Gamma$. Hence 
\[
Z(f)= \bigsqcup \gamma_i (\bbg^o\cap Z(L_{\gamma_i}(f))).
\] 
Let $\iota: \bbq[\bbg]\rightarrow\bbq[\bbg^{\circ}]$ be the homomorphism induced by the restriction map. Then by the assumption on the dimension of $Z(f)$, $\iota(L_{\gamma_i}(f))$ are non-zero, and clearly for any choice of finite sets of prime numbers $S_i$'s and positive integer numbers $r_i$'s, 
\[
\bigsqcup \gamma_i\Gamma^{\circ}_{S_i,r_i}(\iota(L_{\gamma_i}(f)))\subseteq \Gamma_{\cup S_i,\max r_i}(f),
\]
finishing the proof of the lemma.
\end{proof}
\noindent
From this point on we will assume that $\bbg$ is Zariski-connected. Thus all of its derived subgroups are also connected. Let us recall that derived subgroups are defined inductively, $\bbg^{(0)}=\bbg,$ and $\bbg^{(i+1)}=[\bbg^{(i)},\bbg^{(i)}]$. 
\begin{lem}~\label{l:structure}
{\rm (1)} Let $\bbg$ be Levi-semisimple. Then $\bbg^{(i)}$ is also Levi-semisimple,  $\bbg/\bbg^{(i)}$ is unipotent, and $\bbg$ is homeomorphic to $\bbg/\bbg^{(i)}\times\bbg^{(i)}$, as a $\bbq$-variety. In particular, if $\bbg$ is solvable and Levi-semisimple, then it is unipotent.\\
{\rm (2)}  $\Gamma^{(i)}$ is Zariski dense in $\bbg^{(i)}$, for any $i$. 
\end{lem}
\begin{proof} If  $\bbg$ is Levi-semisimple, $\bbg\simeq\bbg_{ss}\ltimes R_u(\bbg)$ as $\bbq$-groups. Therefore for any $i$, $\bbg^{(i)}=\bbg_{ss}\ltimes \bbu_i$, for some $\bbq$-subgroup $\bbu_i$ of $R_u(\bbg)$. Using ~\cite[Theorem 14.2.6]{Sp}, $R_u(\bbg)$ is homeomorphic to $R_u(\bbg)/\bbu_i\times \bbu_i$, as a $\bbq$-variety, and so $\bbg$ is homeomorphic to $\bbg/\bbg^{(i)}\times\bbg^{(i)}$ as a $\bbq$-variety. The other parts are clear.
\end{proof}
\noindent
Since $\bbg$ is connected, for some $i\le \dim \bbg$, $\bbg^{(i)}=\bbg^{(i+1)}$. Let us call 
\[
\bbh=\bbg^{(\dim\bbg)}
\]
 {\it the perfect core} of $\bbg$. Note that the perfect core might be trivial.
 \begin{proof}[Proof of Theorem~\ref{t:Main}.]
 By Lemma~\ref{l:structure}, $\bbq[\bbg]\simeq\bbq[\bbh]\otimes\bbq[\bbg/\bbh]$ and $\bbu=\bbg/\bbh$ is a unipotent $\bbq$-group.  Let $\pi:\bbg\rightarrow\bbg/\bbh$ be the projection map and  $\phi:\bbg/\bbh\rightarrow \bbg$ a $\bbq$-section. Hence $\phi\circ\pi$ is a $\bbq$-morphism from $\bbg$ to itself. Thus there is a finite set $S\rq{}$ of primes such that
 \[
 \phi\circ\pi(\Gamma)\subseteq \gcal(\bbz_{S\rq{}}).
 \]
 Let $\mathcal{H}$ be the Zariski-closure of $\Gamma^{(\bbh)}=\Gamma\cap \bbh$ in $(\mathbb{SL}_N)_{\bbz_{S_0}}$. By Lemma~\ref{l:Zsfree}, we can find a finite set $S$ of primes and a $\bbz_S$-basis of $\bbz_{S_0}[\mathcal{H}]\otimes \bbz_{S}$. So $f$ is mapped to $\sum_{i=1}^m Q_i\otimes a_i$ for some $Q_i\in \bbq[\bbu]$, which means that for any $g\in\bbg$ we have
 \[
 f(g)=\sum_{i=1}^m Q_i(\pi(g))a_i(\phi(\pi(g))^{-1}g)=\sum_{i=1}^m Q_i(\pi(g))L_{\phi(\pi(g))}(a_i)(g)
 .\] 
 Let $P=\gcd_i(Q_i)$ and $P_i=Q_i/P$.  Applying Theorem~\ref{t:unipotent} to $\pi(\Gamma)$, $P$, and $P_i$'s, we can find a positive integer $r$, a finite set $S\rq{}\rq{}$ of prime numbers, and  a Zariski dense subset $X$ of $\pi(\Gamma)$ such that
\begin{enumerate}
\item$P(u)$ has at most  $r$ prime factors in the ring of $S'$-integers, for any $u\in X$.
\item$\Pi(\gcd(P_1(u),\ldots,P_m(u)))\subseteq S'$, for any $u\in X$.
\end{enumerate} 

\noindent 
By the definition, any $u\in X$ is equal to $\pi(\gamma_u)$ for some $\gamma_u\in \Gamma$. We can identify $\bbg$ with $\bbu\times \bbh$ as $\bbq$-varieties via 
 \[
 (u,h)\mapsto \phi(u)h\hspace{1cm}\&\hspace{1cm} g\mapsto (\pi(g),\phi(\pi(g))^{-1} g).
 \]
 For any $u\in X$ and $\gamma_{\bbh}\in \Gamma^{(\bbh)}$, we have that
 \begin{equation}\label{e:ProductStructure}
 \gamma_u\gamma_{\bbh}\mapsto (u, \phi(u)^{-1} \gamma_u\gamma_{\bbh}),
 \end{equation}
 and 
 \begin{equation}\label{e:f-value}
 f(\gamma_u\gamma_{\bbh})=\sum_{i=1}^m Q_i(u) L_{\gamma_u^{-1}\phi(u)}(a_i)(\gamma_{\bbh}).
 \end{equation}
On the other hand, by the above properties and Theorem~\ref{t:Perfect2}, there are a positive integer  $r'$ and a finite set $S\rq{}\rq{}\rq{}$ of prime numbers such that, for any $u\in X$, 
\[
Y_{u}=\Gamma^{(\bbh)}_{r',S\rq{}\rq{}\rq{}}\left(P(u)\sum_{i=1}^m P_i(u) L_{\gamma_u^{-1}\phi(u)}(a_i)\right)=\Gamma^{(\bbh)}_{r',S\rq{}\rq{}\rq{}}\left(\sum_{i=1}^m Q_i(u) L_{\gamma_u^{-1}\phi(u)}(a_i)\right)
\]
 is Zariski dense in $\bbh$.  Therefore $ {\phi(u)^{-1}\gamma_u}Y_{u}$ is also Zariski dense in $\bbh$, for any $u\in X$, as $\gamma_u^{-1}\phi(u)\in\bbh$. Thus  
  \[
 \tilde{X}=\bigsqcup_{u \in X} \{u\}\times {\phi(u)^{-1}\gamma_u}Y_{u}
 \]
 is Zariski dense in $\bbu\times\bbh$, and, by Equations~(\ref{e:ProductStructure}) and (\ref{e:f-value}), we are done. 
 \end{proof}
\section{Effectiveness of our arguments.}
In order to avoid adding unnecessary complications, we did not discuss the effectiveness of our 
argument in the course of the paper. In this section, we address four issues from which one can 
easily verify that our arguments are effective.  
\begin{enumerate}
\item Let $\Gamma$ be the group generated by a finite subset $S$ of $\GL_n(\bbq)$. Let $\bbg$ 
be the Zariski-closure of $\Gamma$ and assume that $\bbg$ is Zariski-connected. Then in the 
course of our arguments (e.g. proof of Proposition~\ref{p:Geometric}), we need to be able to 
compute a presentation for $\bbg$, i.e. compute a finite subset $F$ of $\bbq[\mathbb{GL}_n]$ such 
that $\bbq[\bbg]\simeq \bbq[\mathbb{GL}_n]/\langle F\rangle$. 
\item Computing the irreducible components of a given affine variety and an effective version of 
N\"{o}ether-Bertini theorem is needed in the proof of Proposition~\ref{p:PointsQIrreducible}.
\item The spectral gap of the discrete Laplacian on the Cayley graphs of $\pi_q(\Gamma)$ with respect to $\pi_q(\Omega)$, where $\Omega$ is a symmetric finite generating set of $\Gamma$, if the connected component of the Zariski-closure of $\Gamma$ is perfect.
\item An effective version of Nori's strong approximation theorem is needed in various parts of 
this article, e.g. in understanding the density of the sieve. We need a more or less equivalent 
formulation. To be precise, we need to say that $\pi_q(\Gamma)=\prod_{p|q}\bbg(\f_p)$ if 
$\Gamma$ is a Zariski-dense subgroup of $\bbg$ and $\bbg$ is generated by its $\bbq$-unipotent 
subgroups. 
\end{enumerate}
The first three items are dealt with in~\cite{SV}. \cite[Lemma 62]{SV} gives us the first 
item. In order to get the second item, first we use \cite[Chapter 8.5]{BW}, to compute the primary 
decomposition of the defining ideal of the variety, which gives us the irreducible components. Then 
we use \cite[Theorem 40]{SV} to get an effective version of N\"{o}ether-Bertini theorem. In fact, 
\cite[Theorem 40]{SV} proves an effective version of \cite[Theorem 9.7.7 (i) and Theorem 12.2.4 (iii)]{EGA} which is a generalization of N\"{o}ether-Bertini theorem. The third item is the main result of \cite{SV} and they also show that their result is effective. 

Let $\Gamma$ be the group generated by a finite subset $\Omega$ of $\SL_n(\bbz_{S_0})$. Let 
 $\gcal$ be the Zariski-closure of $\Gamma$ in $(\mathbb{SL}_n)_{\bbz_{S_0}}$ and
 $\gcal_p=\gcal\times \Spec(\f_p)$.

\begin{thm}[Effective version of Theorem 5.4 in \cite{Nor}]\label{t:NoriEffective}
In the above setting, assume that the generic fiber $\bbg$ of $\gcal$ is generated by 
$\bbq$-unipotent subgroups. Then there is a recursively defined function $f$ from finite subsets of 
$\SL_n(\bbq)$ to positive integers such that for any square-free integer $q$ with prime factor 
at least $f(\Omega)$, we have 
\[
 \pi_q(\Gamma) = \prod_{p|q}\gcal_p(\f_p),
\]
where $\Gamma=\langle \Omega\rangle$ and $p$ runs through the prime divisors of $q$.
\end{thm}
\begin{remark}\label{r:NoriEffective}
\begin{enumerate}
\item
In \cite{Nor}, it is said that (the non-effective version of) Theorem~\ref{t:NoriEffective} can be deduced 
from \cite[Theorems A, B and C]{Nor} and the reader is referred to an unpublished manuscript. As 
we need the effective version of this result, we decided to write down a proof of this statement.

\item
In the appendix of~\cite{SV}, the effective versions of \cite[Theorem A, B and C]{Nor} are given. Furthermore \cite[Theorem 40]{SV} provides an effective version of \cite[Theorem 5.1]{Nor}, when
$\bbg$ is perfect.
\end{enumerate}
\end{remark}
\begin{lem}\label{l:FiniteProduct}
Let $\{G_i\}_{i\in I}$ be a finite collection of finite groups such that $G_i$ and $G_{i'}$ do not have 
a (non-trivial) common homomorphic image for $i\neq i'$. Let $H$ be a subgroup of 
$G=\prod_i G_i$. If the projection $\pi_i(H)$ of $H$ to $G_i$ is onto, i.e. $\pi_i(H)=G_i$,
 then $H=G$.
\end{lem}
\begin{proof}
We proceed by induction on the order of $G$. Let $H_i=G_i\cap H$. Since $\pi_i(H)=G_i$, 
$H_i$ is a normal subgroup of $G_i$. Let $G'_i:=G_i/H_i$, $H':=H/\prod_i H_i$ and $G':=\prod_i G_i/H_i$. It is clear that $\pi_i(H')=G'_i$ for any $i$. If $H_i\neq 1$ for some $i$, then $|G|>|G'|$. 
Hence by the induction hypothesis we have that $H'=G'$, which implies that $H=G$ and we are 
done. So without loss of generality we can assume that 
\begin{equation}\label{e:intersection}
H\cap G_i=1,
\end{equation}
for any $i$. For a fixed $i_0$, let $H'_{i_0}$ be the projection of $H$ to $\prod_{i\neq i_0} G_i$. It 
is clear that $\pi_i(H'_{i_0})=G_i$ for any $i\neq i_0$. Hence by the induction hypothesis we have that
\begin{equation}\label{e:projection}
\pi_{I\setminus \{i_0\}}(H)=\prod_{i\neq i_0} G_i,
\end{equation}
where $\pi_{I\setminus \{i_0\}}$ is the projection to $\prod_{i\neq i_0}G_i$. By 
Equations~(\ref{e:intersection}) and (\ref{e:projection}), we have that $G_{i_0}$ is a homomorphic 
image of $\prod_{i\neq i_0} G_i$. Let $N$ be the normal subgroup of $\prod_{i\neq i_0} G_i$ such 
that $(\prod_{i\neq i_0} G_i)/N\simeq G_{i_0}$. Then again by the induction hypothesis there is some
$i_1\neq i_0$ such that $\pi_{i_1}(N)\neq G_{i_1}$. Thus $G_{i_0}$ and $G_{i_1}$ have a (non-trivial) common homomorphic image, which is a contradiction.
\end{proof}
\begin{proof}[Proof of Theorem~\ref{t:NoriEffective}]
Let $\hcal=\dcal^{(\dim \bbg)}(\gcal)$. Then the generic fiber $\bbh$ of $\hcal$ is the perfect core 
of $\bbg$ and $\dcal^{(\dim \bbg)}(\Gamma)\subseteq \Gamma \cap \hcal(\bbz_{S_0})$ is 
Zariski-dense in $\hcal$. Since $\bbh$ is Zariski-connected, by \cite[Lemma 62]{SV}, we can find 
a finitely generated subgroup of $\dcal^{(\dim \bbg)}(\Gamma)$ which is Zariski-dense in $\bbh$. 
Hence by \cite[Theorem 40]{SV} we can compute $p_0$ such that for any $p>p_0$ we have
\begin{equation}\label{e:PerfectCore}
\pi_p(\dcal^{(\dim \bbg)}(\Gamma))=\hcal_p(\f_p),
\end{equation}
where $\hcal_p=\hcal\times_{\Spec(\bbz_{S_0})} \Spec(\f_p)$. On the other hand, 
$\bbu=\bbg/\bbh$ is a unipotent $\bbq$-group and the image $\phi(\Gamma)$ of $\Gamma$  to 
$\bbu$ is a Zariski-dense group. We can compute an embedding of $\bbu$ to 
$(\mathbb{GL}_m)_{\bbq}$ for some $m$. By enlarging $S_0$, if necessary, we can compute the 
Zariski-closure $\ucal$ of $\pi(\Gamma)$ in $(\mathbb{GL}_m)_{\bbz_{S_0}}$ and extend $\phi$ 
to a $\bbz_{S_0}$-homomorphism from $\hcal$ to $\ucal$. Using the logarithmic and exponential 
maps, we can effectively enlarge $p_0$, if necessary, to make sure that
\begin{equation}\label{e:UnipotentQuotient}
\pi_p(\phi(\Gamma))=\ucal_p(\f_p),
\end{equation}
where $\ucal_p=\ucal\times_{\Spec(\bbz_{S_0})}\Spec(\f_p)$ and $p>p_0$. Hence by 
(\ref{e:PerfectCore}) and (\ref{e:UnipotentQuotient}) we have that 
$\pi_p(\Gamma)=\gcal_p(\f_p)$, for any $p>p_0$. By \cite[Lemma 64]{SV} we completely 
understand the composition factors of $\gcal_p(\f_p)$. In particular, if $p$ and $p'$ are primes 
larger than 7, then $\gcal_p(\f_p)$ and $\gcal_{p'}(\f_{p'})$ do not have a (non-trivial) common 
homomorphic image. Thus Lemma~\ref{l:FiniteProduct} finishes the proof.
\end{proof}

\appendix
\section{Heuristics.}
Finding primes or almost primes in very sparse sequences of integers is notoriously difficult and the most believable speculations are based on probabilistic reasoning. Such an argument for Fermat primes $F_n=2^{2^n}+1$ suggest that their number is finite \cite[Page 15]{HW}. Similar arguments have been carried out for other sequences such as Fibonacci numbers \cite{BLMS}. Here we pursue such probabilistic heuristics for tori. 

Let $\Gamma\subseteq \GL_n(\bbq)$ be a (finitely generated)  torus. That is to say it is conjugate to a group of diagonal matrices; there is $g\in \GL_n(K)$, $K$ a number field, such that
\begin{equation}\label{e:Diag}
g^{-1}\Gamma g\subseteq \mathbb{D}(K)=\{\diag(a_1,\ldots,a_n)|\h a_i\in K^{\times}\},
\end{equation}
where 
\[
\diag(a_1,\ldots,a_n)=\left[\begin{array}{ccc}a_1&&\\&\ddots&\\&&a_n\end{array} \right].
\]
The heuristics will show that there is an $f\in \bbq[\bba^{n^2}]$ for which not only does $(\Gamma, f)$ not saturate, but $\Gamma_{r,S}(f)$ is not finite for every $r$ and $S$. To this end, we can assume that $g^{-1}\Gamma g$ is a subgroup of the $S$-units of $K$ for some finite set of places (with $\vare \rightarrow \diag(\sigma_1(\vare),\ldots,\sigma_n(\vare))$ and $\sigma_i$'s the embeddings of $K$). For simplicity we assume that $S$ is empty and since the torsion is finite we can ignore it for our purposes. That is $\Gamma$ is a free abelian group of rank $t\ge 1$ with generators  $\gamma_1,\ldots,\gamma_{t}$ and $\Gamma\subseteq \GL_n(\bbz)$. For $x\in \GL_n(\bbq)$ let $F(x)={\rm Tr}(x^tx)\in \bbq[\bba^{n^2}]$ be its Hilbert-Schmidt norm. It is clear from discreteness property of the log of units map and (\ref{e:Diag}) that for ${\bf m}=(m_1,\ldots,m_t)$ in $\bbz^t$
\begin{equation}\label{e:NormIneq}
A_2^{|{\bf m}|}\ll F(\gamma_1^{m_1}\cdots\gamma_t^{m_t})\ll A_1^{|{\bf m}|},
\end{equation}
where $A_1>A_2>1$ and the implied constants are independent of ${\bf m}$. Fix $\nu>t$ and set 
\begin{equation}\label{e:ShiftProduct}
f(x):=f_1(x)f_2(x)\cdot \cdots \cdot f_{\nu}(x),
\end{equation}
where 
\[
 f_j(x)=F(x)+j,\hspace{1cm} j=1,2,\ldots, \nu. 
\]
The heuristic argument is that for each $m\in \bbz^t$, $F(\gamma_1^{m_1}\cdots\gamma_t^{m_t})$ is a ``random" integer in the range (\ref{e:NormIneq}) and that $f_j(x)$ for $j=1,2,\ldots,\nu$ are independent as far as the number of their prime factors. Actually there may be some small forced prime factors of $f$ but these will only enhance the reasoning below. Now let $r$ be a large integer and for ${\bf m}\in \bbz^t$ let $p({\bf m},r)$ be the probability that an integer in the range (\ref{e:NormIneq}) has at most $r$-prime factors. According to the prime number theorem
\begin{equation}\label{e:Probability}
p({\bf m},r)\ll \frac{[\log(|{\bf m}|+1)]^{r-1}}{|{\bf m}|+1}
\end{equation}
(again the implied constants being independent of ${\bf m}$).

Hence assuming that the number of prime factors of the $f_j$, $j=1,\ldots,\nu$ are independent and that these values are ``random" we see that $p_f({\bf m},r)$, the probability that $f(\gamma_1^{m_1}\cdots\gamma_t^{m_t})$ has at most $r$-prime factors satisfies:
\begin{equation}\label{e:Prob}
p_f({\bf m},r)\ll\frac{[\log(|{\bf m}|+1]^{\nu(r-1)}}{(|{\bf m}|+1)^{\nu}}.
\end{equation}
Hence since $\nu>t$
\begin{equation}\label{e:Conv}
\sum_{{\bf m}\in \bbz^t}p_f({\bf m},r)<\infty
\end{equation}
By the Borel-Cantelli Lemma it follows that the probability that there are infinitely many ${\bf m}$'s for which $f(\gamma_1^{m_1}\ldots\gamma_t^{m_t})$ has at most $r$-prime factors, is zero. That is for any $r$ we should expect that $\Gamma_r(f)$ is finite! 

For a general $\Gamma\subseteq \GL_n(\bbz)$ (or finitely generated in $\GL_n(\bbq)$) if $\bbg=\Zcl(\Gamma)$ is not Levi-semisimple, then there is an onto $\bbq$-homomorphism $\phi:\bbg^{\circ}\rightarrow\bbt$, where $\bbt$ is a non-trivial $\bbq$-torus. Hence $\Lambda=\phi(\Gamma)$ is a finitely generated subgroup of $\bbt(\bbq)$. Thus one can use the heuristics above to show there is an $f\in\bbq[\bbt]$ such that $\Lambda_r(f)$ is finite for any $r$. In particular, as in the proof of Theorem~\ref{t:Main}, $\Gamma_r(f)$ cannot be Zariski-dense in $\bbg$ for any $r$. The conclusion is that if we accept the probabilistic heuristics, then the condition that $\bbg$ be Levi-semisimple in Theorem \ref{t:main2} is necessary if we allow all $f$'s.


\end{document}